\documentclass[11pt]{amsart}

\usepackage{graphicx,amssymb,latexsym,amsfonts,txfonts,amsmath,amsthm}
\usepackage{pdfsync}
\usepackage{amsmath,amscd}
\usepackage[all,cmtip]{xy}
\input epsf

\usepackage{hyperref}
\hypersetup{
    colorlinks=true,       
    linkcolor=blue,          
    citecolor=blue,        
    filecolor=blue,      
    urlcolor=blue           
}

\newcommand{\s}{\vspace{0.2cm}}

\newtheorem{theo}{Theorem}

\newtheorem{coro}{Corollary}

\theoremstyle{remark}
\newtheorem{rema}{\bf Remark}

\begin{document}

\title{Bipartite graphs and their dessins d'enfants}

\author{Ruben A. Hidalgo}
\address{Departamento de Matem\'atica y Estad\'{\i}stica, Universidad de La Frontera.  Temuco, Chile}
\email{ruben.hidalgo@ufrontera.cl}

\thanks{Partially supported by Project Fondecyt 1150003 and Project Anillo ACT1415 PIA-CONICYT}
\subjclass[2010]{Primary 14H57, secondary 05C10, 05C25, 11G32, 30F10}
\keywords{Dessins d'enfants, bipartite graphs, graph embeddings}
\maketitle


\begin{abstract}
Each finite and connected bipartite graph induces a finite collection of non-isomorphic dessins d'enfants, that is, $2$-cell embeddings of it into some  closed orientable surface. 
We describe an algorithm to compute all these dessins d'enfants, together their automorphims group, monodromy group and duality type.
\end{abstract}

\section{Introduction}
A dessin d'enfant, as introduced by Grothendieck in its {\it Esquisse d'un Programme} \cite{Gro}, 
is a $2$-cell embedding of a finite (necessarilly connected) bipartite graph into some closed orientable surface. Readers may consult, for instance \cite{GiGo,JS,LZ,Sch,Wolfart1,Wolfart2} and the references therein. As a consequence of Belyi's theorem \cite{Belyi}, there is a correspondence between dessins d'enfants and non-singular and irreducible projective algebraic curves defined over the field $\overline{\mathbb Q}$ of algebraic numbers. This provides a natural action of the absolute Galois group ${\rm Gal}(\overline{\mathbb Q}/{\mathbb Q})$ on dessins d'enfants, which is known to be faithful \cite{Gro,GiGo, GiGo1, Sch} (even faithful at the level of regular dessin d'enfants \cite{GoJa}). Grothendieck pointed out that such an action should provide information on the internal structure of ${\rm Gal}(\overline{\mathbb Q}/{\mathbb Q})$ codified in terms of simple combinatorial objects. Known Galois invariants of dessins d'enfants are their passports, monodromy groups and group of automorphisms (see Section \ref{Sec:invariantes}). Another Galois invariants have been produced in \cite{Ravi} (exetnding Belyi maps).  So far, there is not known a complete set of Galois invariants, at least to the actual author's knowledge. Recently, in \cite{GGH}, it has been discussed another invariant called the duality type of the dessin:
a dessin d'enfant is dualizable if its faces can be labelled by signs $+$ and $-$, so that adjacent faces have different label, equivalently, the dual graph is bipartite (this notion was originally discussed by Zapponi in \cite{GGH,Zapponi} for the case of clean dessins; he called them orientable ones). 

In recent years there has been an interest on 
dessins d'enfants in the field of supersymmetric gauge and conformal field theories \cite{ACD1,ACD2,Ha-He,He-R,JRR} to mention some of the applications in theoretical physics. In this way, it seems interesting searching for algorithms to provide examples of dessins d'enfants, up to isomorphisms, together some of their Galois invariants.

By the definition, to each dessin d'enfant there is associated a finite and connected bipartite graph. There examples of non-isomorphic dessins d'enfants with isomorphic  associated bipartite graphs (isomorphism as graphs but respecting colouring of vertices). In this paper, given a finite and connected bipartite graph ${\mathcal G}$, we provide an algorithm which permits to construct all  non-isomorphic dessins d'enfants, together their automorphims group, monodromy group and orientability type, whose underlying bipartite graph is isomorphic to ${\mathcal G}$. As an example, for the double-prism bipartite graph shown in Figure \ref{DoublePrism} (see Example \ref{Sec:doble-prisma}), our algorithm determines that
there are $5946$ non-isomorphic dessins d'enfants; $2$ of genus zero (bot are dualizable), $79$ of genus one (only $22$ of them being dualizable), $1849$ of genus two (only $121$ of them being dualizable) and $4016$ of genus three (only $33$ being dualizable). 
Two of these genus one non-isomorphic dessins d'enfants have the same passport $(4^6;2^{12};3^2,4^2,5^2)$ and isomorphic monodromy groups, one of them being chiral and the other being reflexive; so they are not in the same Galois orbit.

\subsection*{The algorithm}
Next, we proceed to describe the algorithm and the main procedure steps.

\underline{Input}:
\begin{enumerate}
\item[(I1)] A finite and connected bipartite graph ${\mathcal G}$ with $e \geq 1$ edges,    
$\alpha$ black vertices $v_{1},\ldots,v_{\alpha}$ and  $\beta$ white vertices  $w_{1},\ldots,w_{\beta}$. 
\item[(I2)] The group $G_{\mathcal G}$ of bipartite graph automorphisms of ${\mathcal G}$ (graph automorphisms preserving vertices of a fixed colour). 
\end{enumerate}

\underline{Output}:
\begin{enumerate}
\item[(O)] A maximal collection of non-isomorphic dessins d'enfants,  whose subjacent bipartite graph is isomorphic to ${\mathcal G}$, together their monodromy group, automorphisms group and duality type.
\end{enumerate}

\underline{Prodedure}:
\begin{enumerate}
\item[(P1)] Fix an enumeration of the $e$ edges of ${\mathcal G}$ with numbers in the set $\{1,\ldots,e\}$ without repeating. 

\item[(P2)] The above enumeration determines a natural injective homomorphism $\theta:G_{\mathcal G} \to {\mathfrak S}_{e}$. One may use the package ``GRAPE" in GAP \cite{GAP} in order to obtain the group of automorphisms of the bipartite graph (at least for clean bipartite graphs); this must be done for the associated edge-graph in order to obtain the action on the edges.

\item[(P3)] For each black vertex $v_{i}$ (respectively, white vertex $w_{j}$) we consider the collection ${\mathcal F}_{{\mathcal G},i}^{b}$ (respectively, ${\mathcal F}_{{\mathcal G},j}^{w}$) of all possible cycles $\sigma_{i}$ (respectively, $\tau_{j}$) of length equal to the degree of such vertex in ${\mathcal G}$ using the numbers at all the edges adjacents to such a vertex. Set 
${\mathcal F}_{\mathcal G}={\mathcal F}_{\mathcal G}^{b} \times {\mathcal F}_{\mathcal G}^{w}$, where
${\mathcal F}_{\mathcal G}^{b}$ is the collection of all the permutations $\sigma=\sigma_{1}\cdots\sigma_{\alpha} \in {\mathfrak S}_{e}$, where $\sigma_{i} \in {\mathcal F}_{{\mathcal G},i}^{b}$,  and 
${\mathcal F}_{\mathcal G}^{w}$ is the collection of all the permutations $\tau=\tau_{1}\cdots\tau_{\beta} \in {\mathfrak S}_{e}$, where $\tau_{j} \in {\mathcal F}_{{\mathcal G},j}^{w}$.

\item[(P4)] Because of the connectivity of ${\mathcal G}$, for each pair $(\sigma,\tau) \in {\mathcal F}_{\mathcal G}$, the group $\langle \sigma, \tau\rangle$ is a transitive subgroup of ${\mathfrak S}_{e}$, so it defines a dessin d'enfant ${\mathcal D}_{(\sigma,\tau)}$ whose associated bipartite  graph is ${\mathcal G}$. 

\item[(P5)] A natural action of $G_{\mathcal G}$ on the set ${\mathcal F}_{\mathcal G}$ is given by
$$G_{\mathcal G} \times {\mathcal F}_{\mathcal G} \to {\mathcal F}_{\mathcal G}: (\phi, (\sigma,\tau)) \mapsto\phi \cdot (\sigma,\tau):=\left(\theta(\phi)^{-1} \sigma \theta(\phi), \theta(\phi)^{-1} \tau \theta(\phi)\right).$$
The multiplication of permutations are from the left to the right as it is done in GAP \cite{GAP}.

\item[(P6)]  As a consequence of Theorem \ref{teomain} (see Section \ref{Sec:mainpart}) we obtain the following facts.
\begin{enumerate}
\item[(I)] If ${\mathcal D}$ is a dessin d'enfant whose underlying bipartite graph is isomorphic to ${\mathcal G}$, then it is isomorphic to ${\mathcal D}_{(\sigma,\tau)}$ for a suitable $(\sigma,\tau) \in {\mathcal F}_{\mathcal G}$. 

\item[(II)] Two pairs in ${\mathcal F}_{\mathcal G}$ define isomorphic dessins d'enfant if and only if they belong to the same $G_{\mathcal G}$-orbit. 

\item[(III)] The group of automorphism of the dessin d'enfant $D_{(\sigma,\tau)}$, where $(\sigma,\tau) \in {\mathcal F}_{\mathcal G}$ is naturally isomorphic to the $G_{\mathcal G}$-stabilizer of $(\sigma,\tau)$.

\item[(IV)]  $D_{(\sigma,\tau)}$ is dualizable if and only if 
there exists a homomorphism $\rho:M=\langle \sigma,\tau\rangle \to {\mathbb Z}_{2}$ so that $\rho(\sigma)=\rho(\tau)=-1$ and ${\rm Stab}_{M}(1)<\ker(\rho)$ \cite{GGH}.

\end{enumerate}

\end{enumerate}

\section{Preliminaries and notations}\label{Sec:dessins}

\subsection{Graphs and their automorphisms}\label{Sec:grafo}
Let us consider a finite  graph ${\mathcal G}=(V,E,\varepsilon)$, where $V$ and $E$ are the finite sets of its vertices and edges, together its 
incident map $\varepsilon:E \to (V \times V)/{\mathfrak S}_{2}$: if $\varepsilon(e)=(v_{1},v_{2})$, then $e$ and the vertices $v_{1}$ and $v_{2}$ are incident (if $v_{1}=v_{2}$, then $e$ is a loop). The degree of $v \in V$ is
$${\rm deg}_{\mathcal G}(v)=2|\varepsilon^{-1}(v,v)|+\sum_{w \in V, v \neq w} |\varepsilon^{-1}(v,w)|.$$

In this paper we will only consider those finite graphs which are connected that is, for every pair of vertices $v_{1}, v_{2} \in V$ there is (finite) collection of edges $e_{1},\ldots,e_{n} \in E$ so that $e_{1}$ and $v_{1}$ are incident, $e_{n}$ and $v_{2}$ are incident and $e_{i}$ with $e_{i+1}$ are both incident to a common vertex. 

An isomorphism of the graphs ${\mathcal G}_{1}=(V_{1},E_{1},\varepsilon_{1})$ and ${\mathcal G}_{2}=(V_{2},E_{2},\varepsilon_{2})$ is a pair $\phi=(\phi_{1},\phi_{2})$, where $\phi_{1}:V_{1} \to V_{2}$ and $\phi_{2}:E_{1} \to E_{2}$ are bijective functions respecting the incident maps, i.e., 
$$[(\phi_{1},\phi_{1})] \circ \varepsilon_{1}=\varepsilon_{2} \circ \phi_{2},$$
where $[(\phi_{1},\phi_{1})]$ is the induced map by $(\phi_{1},\phi_{1}):V_{1} \times V_{1} \to V_{2} \times V_{2}$. In the case ${\mathcal G}_{1}={\mathcal G}_{2}={\mathcal G}$, the isomorphism is a graph-automorphism of ${\mathcal G}$; we denote by $\widehat{G}_{\mathcal G}$ its group of graph automorphisms. 

If we enumerate the $e$ edges of the graph ${\mathcal G}$ with numbers in $\{1,\ldots,e\}$ (without repeating), then to each automorphism $T \in \widehat{G}_{\mathcal G}$ there is associated a permutation $\theta(T) \in {\mathfrak S}_{e}$ (the symmetric group); providing in this way a natural homomorphism $\theta:\widehat{G}_{\mathcal G} \to {\mathfrak S}_{e}$. Let us observe that for a non-trivial graph automorphism $T$ it might be that $\theta(T)$ is the identity. In \cite{Mulase} it was seen that such a pathology only happens for an {\it special graph}, i.t., a graph 
having exactly two vertices and without loops (see Figure \ref{Fig:Mulase}). So, if the graph is not special, then the
homomorphism $\theta$ is injective.

\begin{figure}[htbp]
\begin{center}
\includegraphics[width=2cm]{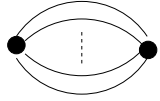}
\caption{{\bf Special graphs: two vertices and $n \geq 1$ edges}}
\label{Fig:Mulase}
\end{center}
\end{figure}

\subsection{Bipartite graphs}
 A bipartite graph is a graph together a colouring of its vertices using two colours, black and white, so that adjacent vertices have different colours. 
The {\it passport} of a bipartite graph ${\mathcal G}$ is the tuple $(a_{1},\ldots,a_{\alpha};b_{1},\ldots,b_{\beta})$,
where $$a_{1} \leq a_{2} \leq \cdots \leq a_{\alpha} \quad \mbox{are the degrees of the black vertices},$$
$$b_{1} \leq b_{2} \leq \cdots \leq b_{\beta} \quad \mbox{are the degrees of the white vertices}.$$

Let us note that, if $e$ is the number of edges of the bipartite graph, then 
$$e=a_{1}+\cdots+a_{\alpha}=b_{1}+\cdots+b_{\beta}.$$

If  $b_{1}=\cdots=b_{\beta}=2$, then the bipartite graph is called {\it clean}.

An element of $\widehat{G}_{\mathcal G}$ either (i) keeps invariant the vertices of a fixed color or (ii) it interchanges the black vertices with the white vertices. We will denote by $G_{\mathcal G}$ the subgroup of $\widehat{G}_{\mathcal G}$ of those automorphisms which sends black (respectively, white) vertices to black (respectively, white) vertices; its elements are called automorphisms of the bipartite graph ${\mathcal G}$.  In most of the cases $\widehat{G}_{\mathcal G}=G_{\mathcal G}$; otherwise $[\widehat{G}_{\mathcal G}:G_{\mathcal G}]=2$ (in this case, the elements in $\widehat{G}_{\mathcal G}-G_{\mathcal G}$ permutes the white vertices with the black ones). We have that the restriction $\theta: {G}_{\mathcal G} \to {\mathfrak S}_{e}$ is then an injective homomorphism.

\s
\begin{rema} If ${\mathcal G}$ is a finite and connected graph, thn we may consider the clean bipartite graph ${\mathcal G}^{clean}$ obtained by colouring all vertices of ${\mathcal G}$ in black and then taking a white vertex in the interior of each of its edges. We may observe that $\widehat{G}_{{\mathcal G}^{clean}}=G_{{\mathcal G}^{clean}}$ if at least one of the vertices of ${\mathcal G}$ has degree at least three. As every graph automorphism of ${\mathcal G}$ induces a bipartite graph automorphism of ${\mathcal G}^{clean}$, there is  an embedding $\chi:\widehat{G}_{\mathcal G} \hookrightarrow G_{{\mathcal G}^{clean}}$; which is surjective if and only if the graph ${\mathcal G}$ has no loops. In particular, if ${\mathcal G}$ has no loops and at least one of its vertices has degree at least three, then $G_{{\mathcal G}^{clean}}$ is naturally isomorphic to $\widehat{G}_{\mathcal G}$. In the case when ${\mathcal G}$ has loops, each loop $e$ provides an extra involution $\phi_{e} \in G_{{\mathcal G}^{clean}}$ that permutes both edges of ${\mathcal G}^{clean}$ contained in $e$ and acts as the identity on all other edges. The group generated by $\chi(\widehat{G}_{\mathcal G})$ and all the involutions $\phi_{e}$ (where $e$ runs over all lops of ${\mathcal G}$) generates the group $G_{{\mathcal G}^{clean}}$. 
\end{rema}

\subsection{Dessins d'enfants: their passports, monodromy groups and group of automorphisms}
Next, we will recall some definitions and facts about dessins d'enfants (the reader may look, for instance, at \cite{GiGo,JS,LZ,Sch,Wolfart1,Wolfart2}). 

\subsubsection{$2$-cell embeddings}
An embedding (or a drawing) of a graph ${\mathcal G}=(V,E,\varepsilon)$ on a closed orientable surface $X$, denoted this by the symbol $\iota:{\mathcal G} \hookrightarrow X$, is a pair of injective functions
$\iota_{1}:V \to X$ and $\iota_{2}:E \to X$, so that $\iota_{2}(e)$ is a an arc homeomorphic to the unit open interval, $\iota_{2}(e_{1}) \cap \iota_{2}(e_{2})=\emptyset$ for $e_{1} \neq e_{2}$, $\iota_{1}(v) \cap \iota_{2}(e)=\emptyset$ (for every $v \in V$ and every $e \in E$), if $v \in V$ is incident to $e$, then $\iota_{1}(v)$ belong to one extreme of $\iota_{2}(e)$ and every extreme of $\iota_{2}(e)$ belongs to $\iota_{1}(V)$. The 
embedding is called a {\it $2$-cell embedding} if each connected component of $X-(\iota_{1}(V) \cup \iota_{2}(E))$ is simply-connected, called the {\it faces} of the embedding. This last condition ensures that ${\mathcal G}$ must be connected.

\s
\begin{rema}
A $2$-cell embedding $\iota:{\mathcal G} \hookrightarrow X$ induces a natural $2$-cell embedding
$\iota:{\mathcal G}^{clean} \hookrightarrow X$. In this way, the study of $2$-cell embedding of finite and connected graphs on closed orientable surfaces is, in some way, equivalent to the study of $2$-cell embeddings of clean bipartite graphs.
\end{rema}

\subsubsection{Dessins d'enfants}
A {\it dessin d'enfant}, as defined by Grothendieck in \cite{Gro}, is a triple ${\mathcal D}=(X,{\mathcal G},\iota)$, where $X$ is a closed orientable surface, ${\mathcal G}$ is a finite bipartite graph (vertices are coloured in black and white) and $\iota:{\mathcal G} \hookrightarrow X$ is a $2$-cell embedding.  Each face of the $2$-cell embedding is called a {\it face} of the dessin. The genus of ${\mathcal D}$ is the genus of $X$.  A face of ${\mathcal D}$ is topologically a polygon of $2r$ sides, where $r \geq 1$; we say that that face has degree $r$. The {\it passport} of ${\mathcal D}$ is the tuple
$(a_{1},\ldots,a_{\alpha};b_{1},\ldots,b_{\beta};c_{1},\ldots,c_{\gamma})$,
where $$a_{1} \leq a_{2} \leq \cdots \leq a_{\alpha} \quad \mbox{are the degrees of the black vertices},$$
$$b_{1} \leq b_{2} \leq \cdots \leq b_{\beta} \quad \mbox{are the degrees of the white vertices},$$
$$c_{1} \leq c_{2} \leq \cdots \leq c_{\gamma} \quad \mbox{are the degrees of the faces}.$$

\s
\begin{rema}
Note that the tuple $(a_{1},\ldots,a_{\alpha};b_{1},\ldots,b_{\beta})$ is the passport of the underlying bipartite graph ${\mathcal G}$.
If $e$ is the number of edges of the dessin d'enfant, then 
$$e=a_{1}+\cdots+a_{\alpha}=b_{1}+\cdots+b_{\beta}=c_{1}+\cdots+c_{\gamma},$$
and, as a consequence of Euler's formula, the genus of $X$ is
$$g=1+\frac{1}{2} \left( e-\alpha-\beta-\gamma\right).$$
\end{rema}

\subsubsection{The monodromy group}
Let ${\mathcal D}=(X,{\mathcal G},\iota)$ be a dessin d'enfant and 
let us denote the black (respectively, white) vertices of ${\mathcal G}$ as $v_{1},\ldots,v_{\alpha}$ (respectively, $w_{1},\ldots,w_{\beta}$).  Let us label the edges of ${\mathcal G}$ with numbers in $\{1,\ldots,e\}$ without repeating. For each black (respectively, white) vertex we chose a cyclic permutation of the edges at the $\iota$-image of that vertex in counterclockwise order (here we are using the orientation of $X$). We then consider, in ${\mathfrak S}_{e}$, the permutation $\sigma$ (respectively, $\tau$) obtained as the product of all these cyclic permutations at black (respectively, white) vertices. The subgroup $\langle \sigma, \tau\rangle$ of ${\mathfrak S}_{e}$, which is a transitive subgroup by the connectivity of ${\mathcal G}$, is called the {\it monodromy group} of ${\mathcal D}$. By the construction, the number of disjoint cycles of $\sigma$ (respectively, $\tau$) is $\alpha$ (respectively, $\beta$) and the number $\gamma$ of faces is equal to the number of disjoint cycles of $\tau\sigma$. 
The integers $a_{1}, \ldots, a_{\alpha}$ are the lengths of the cycles of $\sigma$, the integers 
$b_{1}, \ldots, b_{\beta}$ are the lengths of the cycles of $\tau$ and the integers
$c_{1} \ldots, c_{\gamma}$ are the lengths of the cycles of $\tau\sigma$.

\s
\begin{rema}
If $\sigma, \tau \in {\mathfrak S}_{e}$ are so that $M=\langle \sigma,\tau\rangle$ is 
a transitive subgroup, then $M$ is the monodromy of some dessin d'enfant with $e$ edges. Such a dessin d'enfant is constructed so that 
the disjoint cycles of $\sigma$ correspond to the black vertices, the disjoint cycles of $\tau$ correspond to the white vertices and the disjoint cycles of $\tau\sigma$ corresponds to the faces.
\end{rema}

\subsubsection{Isomorphisms between dessins d'enfants}
Two dessins d'enfants ${\mathcal D}_{1}=(X,{\mathcal G}=(V,E,\varepsilon),\iota=(\iota_{1},\iota_{2}))$ and 
${\mathcal D}_{2}=(\widehat{X},\widehat{\mathcal G}=(\widehat{V},\widehat{E},\widehat{\varepsilon}),\widehat{\iota}=(\widehat{\iota}_{1}, \widehat{\iota}_{2}))$ are called {\it isomorphic} (respectively, {\it non-orientable-isomorphic}) if there is an orientation-preserving (respectively, orientation-reversing) homeomorphism $H:X \to \widehat{X}$ and an isomorphism of bipartite graphs $\phi=(\phi_{1},\phi_{2}):{\mathcal G}_{1} \to {\mathcal G}_{2}$ (i.e., $\phi_{1}$ sends black (respectively, white) vertices to black (respectively, white) vertices, so that $\widehat{\iota}_{1} \circ \phi_{1} = H \circ \iota_{1}$ and $\widehat{\iota}_{2} \circ \phi_{2} = H \circ \iota_{2}$.
We say that the pair $(H,\phi)$, or just $H$ if it is clear in the context, is an {\it isomorphism} (respectively, {\it non-orientable isomorphism}) of the above two dessins d'enfants. 
In terms of the monodromy groups, the isomorphism of the dessins d'enfants ${\mathcal D}_{1}$ and ${\mathcal D}_{2}$ (with the same number $e$ of edges) can be stated as follows. Let us enumerate the edges of each dessin and consider the corresponding monodromy groups $\langle \sigma_{1}, \tau_{1} \rangle$ and $\langle \sigma_{2}, \tau_{2} \rangle$. Then ${\mathcal D}_{1}$ and ${\mathcal D}_{2}$ are isomorphic if and only if there is a permutation $\eta \in {\mathfrak S}_{e}$ so that $\eta \sigma_{1} \eta^{-1}=\sigma_{2}$ and $\eta \tau_{1} \eta^{-1}=\tau_{2}$ (see, for instance, \cite{GiGo}). 

Two dessins are called {\it chirals} if they are non-orientable isomorphic but they are not isomorphic.

\s
\begin{rema}
Observe that the passport of (orientable or non-orientable) isomorphic dessins d'enfants is the same. There known examples of non-isomorphic dessins d'enfant with the same passport and of  non-isomorphic dessins d'enfants (with same number of edges $e$) with the same monodromy group (this happens since we may have $\langle \sigma_{1},\tau_{1}\rangle=\langle \sigma_{2},\tau_{2}\rangle$). 
\end{rema}

\subsubsection{Automorphisms of dessins d'enfants}
An {\it automorphism} of a dessin d'enfant ${\mathcal D}=(X,{\mathcal G},\iota)$ is given by any self-isomorphism $(H,\phi)$ of it (in this definition, $H$ may or not preserve the orientation of $X$).
The {\it group of automorphisms} of ${\mathcal D}$ is denoted by the symbol ${\rm Aut}({\mathcal D})$.

There is a subgroup (of index at most two) ${\rm Aut}^{+}({\mathcal D})$, called its {\it group of orientation-preserving automorphisms}, which corresponds to those self-isomorphisms $(H,\phi)$ where $H$ is orientation-preserving. In most of the cases we have that ${\rm Aut}^{+}({\mathcal D})={\rm Aut}({\mathcal D})$ (the dessin has no orientation-reversing automorphisms). A dessin d'enfant admitting an orientation-reversing automorphism is called {\it reflexive}. 

There is a natural homomorphism $\rho:{\rm Aut}({\mathcal D}) \to G_{\mathcal G}$ (which is injective if the graph cannot be embedded in a circle); its restriction 
$\rho:{\rm Aut}^{+}({\mathcal D}) \to G_{\mathcal G}$  is always injective.

\s
\begin{rema}
(1) A labelling of the $e$ edges of ${\mathcal G}$, as before, provides  (i) a monodromy group $M=\langle \sigma, \tau \rangle < {\mathfrak S}_{e}$ of ${\mathcal D}$ and (ii) an injective homomorphism
$\theta:G_{\mathcal G} \to {\mathfrak S}_{e}$. It can be seen that the isomorphic image $\theta(\rho({\rm Aut}^{+}({\mathcal D}))$ is 
the centralizer of $M$. (2) The group $\langle \sigma^{-1},\tau^{-1}\rangle$ is the monodromy group of some dessin d'enfant $\overline{{\mathcal D}}$; called the {\it conjugated dessin} of ${\mathcal D}$. 
The Riemann surface structure on $X$ defined by $\overline{{\mathcal D}}$ is the conjugated of that defined by ${\mathcal D}$. These two conjugated dessins are isomorphic if and only if ${\mathcal D}$ admits an anticonformal automorphism (i.e., if the dessin is reflexive); otherwise, ${\mathcal D}$ and $\overline{{\mathcal D}}$ is a chiral pair.
\end{rema}

\subsection{Galois actions on dessins d'enfants}\label{Sec:invariantes}
A {\it Belyi pair} is a pair $(S,{\mathfrak B})$, where $S$ is a closed Riemann surface and ${\mathfrak B}:S \to \widehat{\mathbb C}$ is a non-constant meromorphic map whose branch values are contained in the set $\{\infty,0,1\}$; in this case,  $S$ is a {\it Belyi curve} and that ${\mathfrak B}$ is a {\it Belyi map} for $S$. As consequence of Weil's descent theorem \cite{Weil}, every Belyi pair can be defined over the field $\overline{\mathbb Q}$ of algebraic numbers. On the other direction, Belyi's theorem \cite{Belyi} asserts that if $S$ is a closed Riemann surface which can be defined by an algebraic curve over $\overline{\mathbb Q}$, then $S$ is a Belyi curve (and its has a Belyi map defined over ${\mathbb Q}$). In particular, there is a natural action of the absolute Galois group ${\rm Gal}(\overline{\mathbb Q}/{\mathbb Q})$ on Belyi pairs which is known to be faithfull; in genus one this was observed by Grothendieck, in genus zero by Schneps  \cite{Sch}, in hyperelliptic Belyi pairs by Gonz\'alez-Diez and Girondo \cite{GiGo, GiGo1} and in the non-hyperelliptic case in \cite{HJ}. 
It is well known that there is a bijective correspondence between the equivalence classes of the following  objects  (see e.g \cite{GiGo, Gro, JW})
\begin{enumerate}
\item Dessins d'enfants  ${\mathcal D}=(X,{\mathcal G}, \iota)$ with $e$ edges;
\item Belyi pairs $(S, {\mathfrak B})$ of degree $e$; 
\item Subgroups $\Gamma$ of index $e$ of triangle groups $\Delta(l,m,n)=\langle x,y: x^{l}=y^{m}=(yx)^{n}=1\rangle$;
\item Pairs of  permutations  $\sigma, \tau  \in {\mathfrak S}_{e}$  generating transitive subgroups $M=\langle \sigma, \tau \rangle$ of ${\mathfrak S}_{e}$.
\end{enumerate}

The link between these four classes of objects is made as follows. Given a Belyi pair $(S, {\mathfrak B})$ one gets a dessin d'enfant by setting $X = S$, ${\mathcal G}= {\mathfrak B}^{-1}([0,1])$ and black (respectively, white) vertices are provided by ${\mathfrak B}^{-1}(0)$ (respectively, ${\mathfrak B}^{-1}(1)$.
A Fuchsian  group $\Gamma$ as in (3)  defines a Belyi function ${\mathfrak B}$ by simply considering the natural projection  ${\mathbb H}^{2}/ \Gamma \to  {\mathbb H}^{2}/ \Delta(l,m,n)$.
Finally, two permutations  $\sigma, \tau$,  of orders $l$ and $m$  as in (4),  with $n={\rm ord}(\tau \sigma)$,  give rise to a Fuchsian  group  $\Gamma$ as in (3) by considering the epimorphism  
$\omega: \Delta(l,m,n) \to {\mathfrak S}_{e}$  obtained   by sending the generators $x,y$   of  $\Delta(l,m,n)$    to $\sigma$ and $\tau$, respectively, and setting   $\Gamma =\omega^{-1}({\rm Stab}_{M}(1))$.
We use freely this correspondence and we will speak of the Belyi pair, the Fuchsian group and the permutation representation (or monodromy group) defining (or associated to) a dessin ${\mathcal D}$.

As a consequence of the previous correspondence, there is a faithful action of ${\rm Gal}(\overline{\mathbb Q}/{\mathbb Q})$ on the collection of dessins d'enfants.
If ${\mathcal D}$ is a dessin d 'enfant and $\sigma \in {\rm Gal}(\overline{\mathbb Q}/{\mathbb Q})$, then ${\mathcal D}^{\sigma}$ will denote the image of the dessin d'enfant ${\mathcal D}$ by the action of $\sigma$. The following properties of the dessin d'enfant are invariant under the action of the absolute Galois group \cite{GiGo}:
(1) the number of edges, (2) the passport, (3) the genus, (4) the monodromy group and (5) the group of automorphisms. 

\s
\begin{rema}
A consequence, the passport of the underlying bipartite graph is also kept invariant under the action of the absolute Galois group; but not necessarily the isomorphism type of the graph.
\end{rema}

\s

Unfortunately, the above Galois-invariants are not in general enough to decide if two dessins d'enfants belong to the same Galois orbit (examples can be found in \cite{GiGo}). Another Galois invariants have been produced in \cite{Ravi}.  

A dessin d'enfant ${\mathcal D}=(X,{\mathcal G},\iota)$ is called {\it dualizable} if we may paint its faces in two different colours so that adjacent faces have different colours; equivalently, the dual graph defines also a dessin d'enfant on $X$. In \cite{GGH} it has been shown that duality type of the dessin is another Galois invariant of a dessin d'enfant.

\s
\begin{theo}[\cite{GGH}]
Let ${\mathcal D}=(X,{\mathcal G},\iota)$ be a dessin d'enfant with monodromy group $M=\langle \sigma,\tau \rangle$. Then ${\mathcal D}$ is dualizable if and only if 
there exists a homomorphism $\rho:M \to {\mathbb Z}_{2}$ so that $\rho(\sigma)=\rho(\tau)=-1$ and ${\rm Stab}_{M}(1)<\ker(\rho)$.
\end{theo}

\s
\begin{rema}
If $M={\mathfrak S}_{e}$, for $e \geq 3$, as there is a unique surjective homomorphism $\rho:{\mathfrak S}_{e} \to {\mathbb Z}_{2}$ (its kernel being the alternating group ${\mathcal A}_{e}$), then the duality type of the dessin d'enfant is equivalent to have $\sigma,\tau \in {\mathfrak S}_{e}-{\mathcal A}_{e}$ and ${\rm Stab}_{M}(1)<{\mathcal A}_{e}$
\end{rema}

\subsubsection{Field of moduli and field of definition}
The {\it field of moduli} of a dessin d'enfant ${\mathcal D}$ is the fixed field of the subgroup of ${\rm Gal}(\overline{\mathbb Q}/{\mathbb Q})$ formed of those field automorphisms $\sigma$ so that ${\mathcal D}^{\sigma}$ is isomorphic to ${\mathcal D}$. A {\it field of definition} of a dessin d'enfant is a subfield $K$ of ${\mathbb C}$ so that there is a Belyi pair defined over $K$ defining the dessin. 
It is well known that every field of definition of a dessin contains its field of moduli and that the intersection of all of the fields of definitions is exactly the field of moduli \cite{Koizumi} (at this point observe that this last fact will be in general false if we only consider fields of definitions inside $\overline{\mathbb Q}$).

\section{Dessins d'enfants defined by a bipartite graph}\label{Sec:mainpart}
In this section we proceed to describe the theoretical bases for our algorithm (see Theorem \ref{teomain}) which permits to obtain, up to isomorphisms, those dessins d'enfants having the same underlying bipartite graph.

\subsection{The starting data}
Let ${\mathcal G}$ be a (connected and finite) bipartite graph, whose black (respectively, white) vertices are $v_{1},\ldots,v_{\alpha}$ (respectively, $w_{1},\ldots,w_{\beta}$) and let $G_{\mathcal G}$ be its bipartite graph automorphisms. Let us fix a labelling of the $e$ edges of ${\mathcal G}$ with numbers in $\{1,\ldots,e\}$ without repetitions. As seen in Section \ref{Sec:grafo}, this enumeration provides an injective homomorphism $\theta:G_{\mathcal G} \to {\mathfrak S}_{e}$.

\subsection{The collection ${\mathcal F}_{\mathcal G}$}
For each black vertex $v_{i}$ (respectively, white vertex $w_{j}$) we consider the collection ${\mathcal F}_{{\mathcal G},i}^{b}$ (respectively, ${\mathcal F}_{{\mathcal G},j}^{w}$) of all possible cycles $\sigma_{i}$ (respectively, $\tau_{j}$) of length equal to the degree $d_{v_{i}}$ (respectively, $d_{w_{j}}$) using all the labels at the edges at such a vertex. Clearly,  $\#({\mathcal F}_{{\mathcal G},i}^{b})=({\rm deg}_{\mathcal G}(v_{i})-1)!$ and $\#({\mathcal F}_{{\mathcal G},j}^{w})=({\rm deg}_{\mathcal G}(w_{j})-1)!$. Let ${\mathcal F}_{\mathcal G}^{b}$ (respectively, ${\mathcal F}_{\mathcal G}^{w}$) 
be the collection of all the permutations $\sigma=\sigma_{1}\cdots\sigma_{\alpha} \in {\mathfrak S}_{e}$, where $\sigma_{i} \in {\mathcal F}_{{\mathcal G},i}^{b}$ (respectively, $\tau=\tau_{1}\cdots\tau_{\beta} \in {\mathfrak S}_{e}$, where $\tau_{j} \in {\mathcal F}_{{\mathcal G},j}^{w}$). Set 
${\mathcal F}_{\mathcal G}={\mathcal F}_{\mathcal G}^{b} \times {\mathcal F}_{\mathcal G}^{w}$, whose cardinality is
$$(*) \quad N({\mathcal G})=\#{\mathcal F}_{\mathcal G}=\left(\prod_{j=1}^{\alpha} ({\rm deg}_{\mathcal G}(b_{j})-1)! \right)
\left( \prod_{j=1}^{\beta} ({\rm deg}_{\mathcal G}(w_{j})-1)! \right)$$

The connectivity of ${\mathcal G}$ asserts that, for each $(\sigma,\tau) \in {\mathcal F}_{\mathcal G}$, the subgroup $\langle \sigma, \tau\rangle$ of ${\mathfrak S}_{e}$ is transitive; so it is the monodromy group of a dessin d'enfant ${\mathcal D}_{(\sigma,\tau)}$, whose underlying bipartite graph is ${\mathcal G}$.  Let us observe that every dessin d'enfant whose underlying graph is isomorphic to ${\mathcal G}$ must be isomorphic to one of the dessins d'enfants defined by an element of ${\mathcal F}_{\mathcal G}$.

\s
\begin{rema}[Clean bipartite graphs]
If ${\mathcal G}^{clean}$ is the clean bipartite graph associated to a finite and connected graph ${\mathcal G}$, then 
in the definition of the collection ${\mathcal F}_{{\mathcal G}^{clean}}$ we may only consider the first coordinate $\sigma$'s as the second one $\tau$ is uniquely determined. 
\end{rema}

\subsection{Action of the group $G_{\mathcal G}$ on the collection ${\mathcal F}_{\mathcal G}$}
If $(\sigma,\tau) \in {\mathcal F}_{\mathcal G}$ and $\phi \in G_{\mathcal G}$, then  (as $\phi$ preserves the colours and incidences of edges) $\left(\theta(\phi)^{-1} \sigma \theta(\phi), \theta(\phi)^{-1} \tau \theta(\phi)\right) \in {\mathcal F}_{\mathcal G}$. This provides a natural action of $G_{\mathcal G}$ over the collection ${\mathcal F}_{\mathcal G}$.

\s
\begin{theo}\label{teomain}
\begin{enumerate}
\item Two pairs in ${\mathcal F}_{\mathcal G}$ define isomorphic dessins d'enfants if and only if they belong to the same $\theta(G_{\mathcal G})$-orbit. In particular, the cardinality of the quotient set ${\mathcal F}_{\mathcal G}/\theta(G_{\mathcal G})$ is equal to the number of different isomorphic dessins d'enfant having ${\mathcal G}$ as underlying bipartite graph. 

\item The $\theta(G_{\mathcal G})$-stabilizer of the dessin d'enfant ${\mathcal D}$ whose monodromy group is $\langle \sigma, \tau\rangle$, for  $(\sigma,\tau) \in {\mathcal F}_{\mathcal G}$, is equal to $\theta(\rho({\rm Aut}^{+}({\mathcal D})))$ and it is the group of automorphisms of it.
\end{enumerate}
\end{theo}
\begin{proof}
Let $(\sigma_{1},\tau_{1}), (\sigma_{2},\tau_{2}) \in {\mathcal F}_{\mathcal G}$ and let the corresponding dessins d'enfants be 
$(X_{1},{\mathcal G},\iota_{1})$ and $(X_{2},{\mathcal G},\iota_{2})$. If these are isomorphic dessins d'enfant, then there is an orientation-preserving  homeomorphism $F:X_{1} \to X_{2}$ and an automorphism $\phi=(\phi_{1},\phi_{2})$ of ${\mathcal G}$ (as a bipartite graph) so that $\iota_{2} \circ\phi = H \circ \iota_{1}$. Conversely, let us assume there is an automorphism $\phi$ of the bipartite graph ${\mathcal G}$ so that $\theta(\phi)$ conjugates $\sigma_{1}$ to $\sigma_{2}$ and $\tau_{1}$ to $\tau_{2}$. Then it also conjugates $\tau_{1}\sigma_{1}$ to $\tau_{2}\sigma_{2}$. This permits to construct a an orientation-preserving homeomorphism $H:X_{1} \to X_{2}$ so that $(H,\phi)$ is an automorphism of ${\mathcal G}$. This provides part (1). Part (2) is consequence of part (1).
\end{proof}

\s
\begin{rema}[Chirality/reflexivity] 
Let $(\sigma_{1},\tau_{1}), (\sigma_{2},\tau_{2}) \in {\mathcal F}_{\mathcal G}$ and let us assume that $(\sigma_{2}^{-1},\tau_{2}^{-1})$ belongs to the $G_{\mathcal G}$-orbit of $(\sigma_{1},\tau_{1})$. Then the following holds.
\begin{enumerate}
\item If $(\sigma_{1},\tau_{1})$ and $(\sigma_{2},\tau_{2})$ belong to different orbits (i.e. they are non-isomorphic dessins), then ${\mathcal D}_{(\sigma_{1},\tau_{1})}$ and ${\mathcal D}_{(\sigma_{2},\tau_{2})}$  form a chiral pair.
\item If $(\sigma_{1},\tau_{1})$ and $(\sigma_{2},\tau_{2})$ belong to the same orbit (i.e. they are isomorphic dessins), then ${\mathcal D}_{(\sigma_{1},\tau_{1})}$ is reflexive.
\end{enumerate}
\end{rema}

\s
\begin{coro}
If ${\mathcal G}$ is a bipartite graph with trivial group of automorphisms (as bipartite graph), then the number of different 
isomorphic dessins d'enfant admitting ${\mathcal G}$ as its bipartite graph is $N({\mathcal G})$.
\end{coro}

\s
\begin{rema}[On Wilson's operations]
In \cite{Wilson} there were defined the Wilson's operations on dessins d'enfants. These operations are defined as follows. Let fix $r,s$ be positive  integers so that $r$ (respectively, $s$) is co-prime to all degrees of black (respectively, white) vertices of the bipartite graph ${\mathcal G}$. The Wilson's operation $H(r,s):{\mathcal F}_{\mathcal G} \to {\mathcal F}_{\mathcal G}$ is defined by sending 
$(\sigma,\tau) \in {\mathcal F}_{\mathcal G}$ to the new pair $(\sigma_{r,s},\tau_{r,s}) \in {\mathcal F}_{\mathcal G}$ where $\sigma_{r,s}=\sigma^{r}$ and $\tau_{r,s}=\tau^{s}$.  A graph theoretic characterization of certain quasiplatonic curves defined over cyclotomic fields, based on Wilson's operations on maps, is developed in \cite{J-S-W}
\end{rema}

\subsection{A remark on the graph genus}
Let ${\mathcal G}$ be a finite and connected graph. The graph genus $\mu({\mathcal G})$ of ${\mathcal G}$ is the minimal genus of 
a closed orientable surface on which there is an embedding of ${\mathcal G}$. In \cite{Youngs} it has been seen that such a minimal genus embedding is in fact a $2$-cell embedding with a maximal number of faces. The determination of $\mu({\mathcal G})$ seems to be a difficult task, but in the same paper an algorithm to determine it was obtained (claimed that such an algorithm is lengthy). It is known that the problem of finding the graph genus is NP-hard and the problem of determining whether an n-vertex graph has genus $g$ is NP-complete \cite{Carsten}. For some types of graphs (vertex transitive ones) some information is known; for instance, $\mu(C_{n})=0$, $\mu(K_{n})=\lceil(n-3)(n-4)/12\rceil$ (see, \cite{R-Y}), $\mu(K_{n,n})=\lceil(n-2)^{2}/4\rceil$ (see, \cite{Ringel}).
Our algorithm can be used to compute $\mu({\mathcal G})$ as follows. Assume the number of edges of the graph is $e$ and the number of its vertices is $\alpha$.
We consider its associated clean bipartite graph ${\mathcal G}^{clean}$, an enumeration of its $2e$ edges and the corresponding collection ${\mathcal F}_{{\mathcal G}^{clean}}$. As there is only one possible permutation $\tau$ (this being a product of transpositions), we may just consider only the permutations $\sigma$. Now, for each $\sigma$, we consider the product permutation $\tau\sigma$ and we let $\gamma$ be the number of its disjoint cycles. If $\gamma^{max}$ is the maximal possible value of $\gamma$, then the minimal genus of ${\mathcal G}$ is
$$\mu({\mathcal G})=1+e-(\alpha+\gamma^{max})/2.$$
Similarly, if we let $\gamma^{min}$ be the minimal value for $\gamma$, then the maximal $2$-cell embedding genus of ${\mathcal G}$ is 
$$\nu({\mathcal G})=1+e-(\alpha+\gamma^{min})/2.$$

\section{Examples: some classical bipartite graphs}\label{Sec:ejemplos}
In this section, we consider some well known bipartite graphs and we use our algorithm to compute the number of corresponding non-isomorphic dessins d'enfants together their monodromy group,  automorphism group and  duality type. 

\subsection{Example 1}
Let $A_4$ be the clean bipartite graph obtained from the one skeleton of the tetrahedra (its vertices being the black vertices and as white vertices we use a middle point of each side; see Figure \ref{Fig2}). In this case $e=12$, $\alpha=4$ and $\beta=6$. In \cite{GiGo1} two non-isomorphic dessins d'enfants (one of genus zero and the other of genus one) admitting $A_4$ as bipartite graph were provided. 
Our algorithm permits to see that there are exactly three non-isomorphic dessins d'enfants with such property; the missing one has also genus one. In this case, $G_{A_{4}}\cong {\mathfrak S}_{4}$ and $\theta(G_{A_{4}})=\langle \eta_{1}, \eta_{2}, \eta_{3} \rangle$,
where $$\eta_{1}=(1,2,3)(4,5,6)(7,8,9)(10,11,12),\; \eta_{2}=(1,4)(8,11)(5,9)(2,12)(3,7)(6,10),$$
$$\eta_{3}=(2,8,9,4)(5,11,12,1)(3,10,6,7).$$

\begin{figure}[htbp]
\begin{center}
\includegraphics[width=3cm]{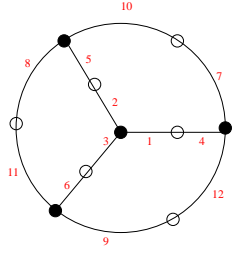}
\caption{{\bf Bipartite graph with passport $(3^4;2^6)$}}
\label{Fig2}
\end{center}
\end{figure}

The set ${\mathcal F}_{A_{4}}$ has $16$ elements and 
${\mathcal F}_{A_{4}}/\theta(G_{A_{4}})$ has three elements, these are represented by the pairs
$(\sigma_{1},\tau), (\sigma_{2},\tau), (\sigma_{3},\tau)$,
where
$$\tau=(1,4)(2,5)(3,6)(7,10)(8,11)(9,12),\;
\sigma_{1}=(1,2,3)(4,7,12)(5,8,10)(6,9,11),$$
$$\sigma_{2}=(1,2,3)(4,7,12)(5,8,10)(6,11,9),\;
\sigma_{3}=(1,2,3)(4,12,7)(5,10,8)(6,11,9),$$
that is, there are exactly three non-isomorphic dessins d'enfant with $A_{4}$ as bipartite graph. The $G_{A_4}$-orbit of $(\sigma_1,\tau)$ has length $8$, the one of $(\sigma_2,\tau)$ has length $6$ and the one of $(\sigma_{3},\tau)$ has length $2$.

The dessin d'enfant defined by the pair $(\sigma_{1},\tau)$ has genus $g=1$, its monodromy group is isomorphic to 
$(({\mathbb Z}_{3} \times ({\mathbb Z}_{3}^{2} \rtimes {\mathbb Z}_{2}))\rtimes {\mathbb Z}_{2})\rtimes {\mathbb Z}_{3}$ and its group of automorphisms is isomorphic to ${\mathbb Z}_{3}$ (see the left in Figure \ref{FigD1}).
The dessin d'enfant defined by the pair $(\sigma_{2},\tau)$ has genus $g=1$, its monodromy group is isomorphic to 
$({\mathbb Z}_{4}^{2} \rtimes {\mathbb Z}_{3})\rtimes {\mathbb Z}_{2}$ and its group of automorphisms is isomorphic to $D_{4}$ (the dihedral group of order $8$)  (see the right in Figure \ref{FigD1}).
The dessin d'enfant defined by the pair $(\sigma_{3},\tau)$ has genus $g=0$, is regular and its monodromy group (isomorphic to the group of automorphisms) is isomorphic to ${\mathcal A}_{4}$.
All these three dessins are reflexive ones and are not dualizable (as there are vertices of odd degree).

\begin{figure}[htbp]
\begin{center}
\includegraphics[width=3cm]{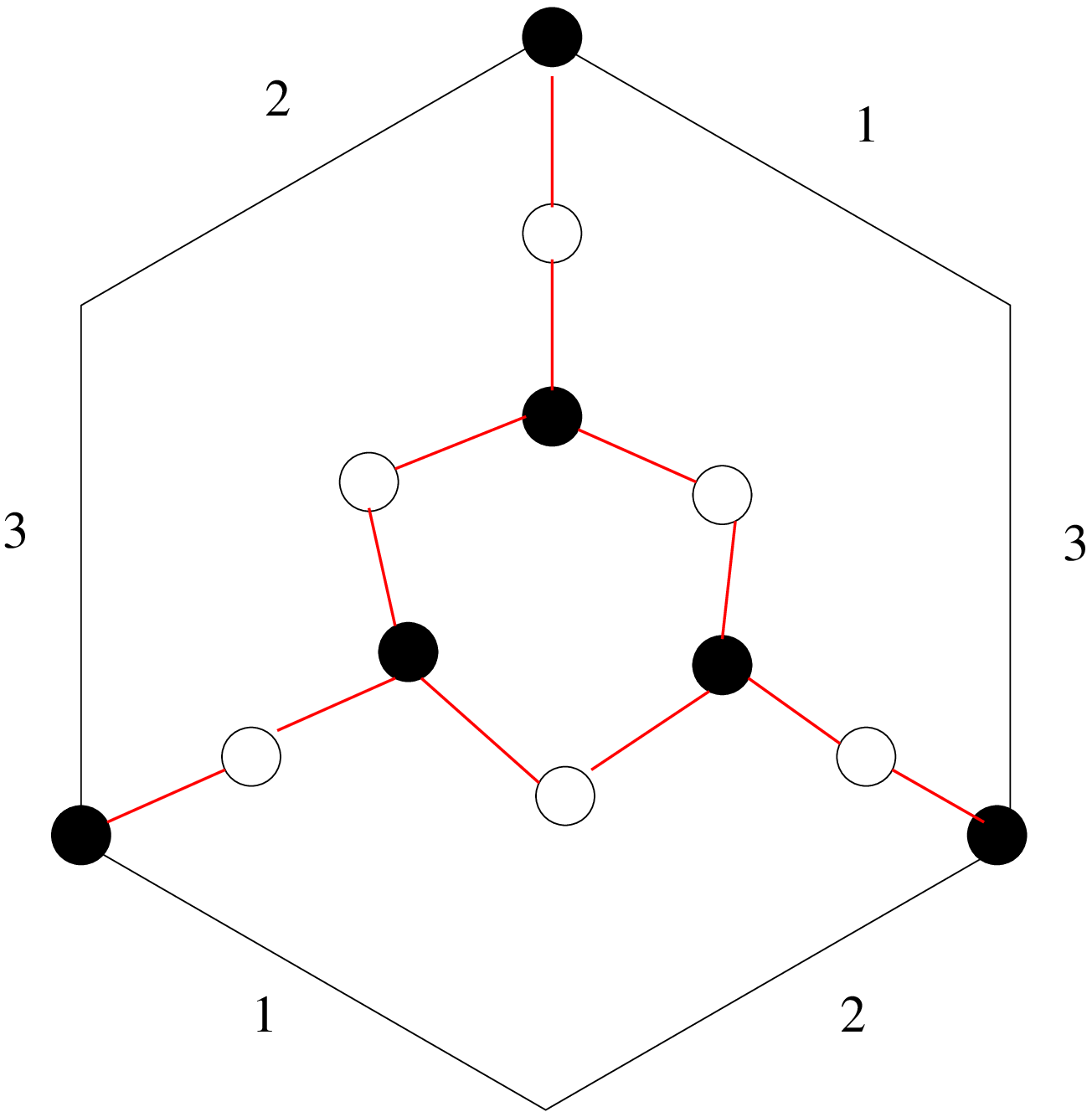}
\includegraphics[width=4cm]{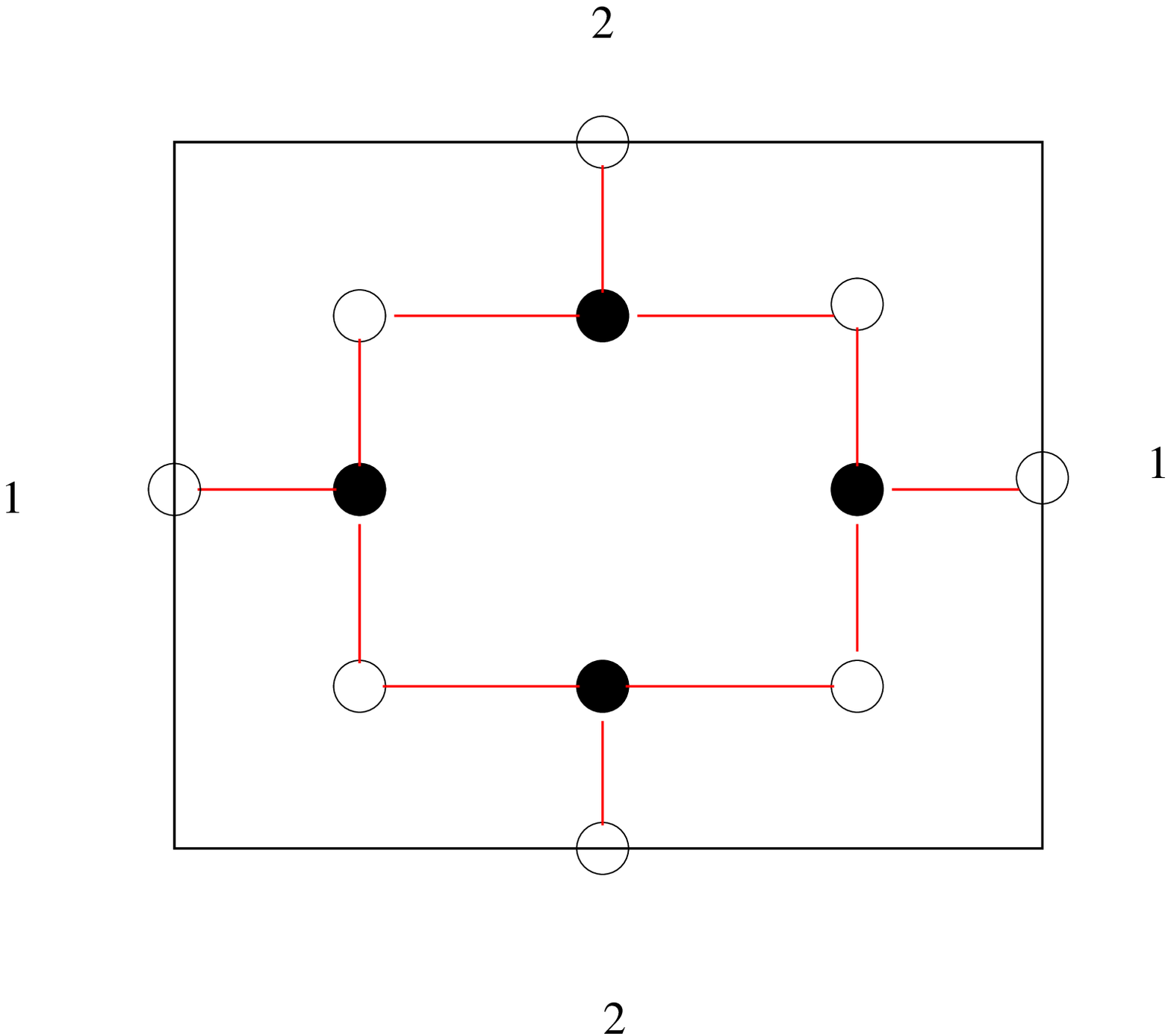}
\caption{{\bf Dessins d'enfants of genus $g=1$ with $A_4$ as bipartite graph}}
\label{FigD1}
\end{center}
\end{figure}

\subsection{Example 2}\label{ejemplo2}
Let $K_{3,3}^{clean}$ be the clean bipartite graph associated to $K_{3,3}$ (its black vertices are the vertices of $K_{3,3}$ and the white vertices are given at middle points of all of the edges), so $e=18$, $\alpha=6$ and $\beta=9$. In this case, $G_{K_{3,3}^{clean}}\cong {\mathfrak S}_{3}^{2} \rtimes {\mathbb Z}_{2}$ and $\theta(G_{A_{4}})=\langle \eta_{1}, \eta_{2}, \eta_{3} \rangle$,
where $$\eta_{1}=(1,4,7)(2,5,8)(3,6,9)(12,11,10)(15,14,13)(18,17,16),$$
$$\eta_{2}=(1,12)(2,11)(3,10)(4,15)(5,14)(6,13)(7,18)(9,16)(8,17),$$
$$\eta_{3}=(1,4)(2,5)(3,6)(12,11)(15,14)(18,17).$$

The set ${\mathcal F}_{K_{3,3}^{clean}}$ has $64$ elements and 
${\mathcal F}_{K_{3,3}^{clean}}/\theta(G_{K_{3,3}^{clean}})$ has three elements, these are represented by the pairs
$(\sigma_{1},\tau), (\sigma_{2},\tau), (\sigma_{3},\tau)$,
where
$$\tau=(1,12)(2,15)(3,18)(4,11)(5,14)(6,17)(9,16)(7,10)(8,13),$$
$$\sigma_{1}=(1,2,3)(4,5,6)(7,8,9)(10,11,12)(13,14,15)(16,17,18),$$ 
$$\sigma_{2}=(1,2,3)(4,5,6)(7,8,9)(10,11,12)(13,14,15)(16,18,17),$$
$$\sigma_{3}=(1,2,3)(4,5,6)(7,9,8)(10,11,12)(13,14,15)(16,18,17),$$
so, there are exactly three non-isomorphic dessins d'enfant with $K_{3,3}^{clean}$ as bipartite graph (i.e., there exactly three topologically types of embedding the graph $K_{3,3}$ in an orientable closed surface as a map). The $G_{K_{3,3}^{clean}}$-orbit of $(\sigma_1,\tau)$ has length $4$, the one of $(\sigma_2,\tau)$ has length $24$ and the one of $(\sigma_{3},\tau)$ has length $36$.
The dessin d'enfant defined by the pair $(\sigma_{1},\tau)$ has genus $g=1$, is regular and its monodromy group (isomorphic to 
the group of automorphisms) is isomorphic to ${\mathbb Z}_{3} \times {\mathfrak S}_{3}$. This regular dessin d'enfant is defined over the Fermat curve $x^3+y^3+z^3=0$; the factor ${\mathbb Z}_{3}$ is generated by the automorphism
$a[x:y:z]:=[x:y:\omega_{3}z]$ and the factor $ {\mathfrak S}_{3}$ is generated by the involution $b[x:y:z:]=[y:x:z]$ and the order three automorphism
$c[x:y:z:]=[\omega_{3}x:\omega_{3}^{2}y:z]$, where $\omega_{3}=e^{2 \pi i/3}$.
The dessin d'enfant defined by the pair $(\sigma_{2},\tau)$ has genus $g=2$, its monodromy group is isomorphic to 
$(({\mathbb Z}_{3} \times ({\mathbb Z}_{3}^{2} \rtimes {\mathbb Z}_{3})\rtimes {\mathbb Z}_{3})\rtimes {\mathbb Z}_{2}$ and its group of automorphisms is isomorphic to ${\mathbb Z}_{3}$. This corresponds to the genus two Riemann 
surface defined by $y^{3}=(x-1)(x^2+x+1)^{2}$.
The dessin d'enfant defined by the pair $(\sigma_{3},\tau)$ has genus $g=1$, its monodromy group is isomorphic to 
${\mathfrak S}_{9}$ and its group of automorphisms is isomorphic to ${\mathbb Z}_{2}$.
All the above dessins are reflexive and are not dualizable (as there are vertices of odd degree).

\subsection{Example 3: Frucht's graph}
The first example of a finite and connected graph $F$ with trivial group of automorphisms was provided by R. Frucht in \cite{Frucht} (see Figure \ref{FigFrucht1}). The associated clean bipartite graph $F^{clean}$ is shown in Figure \ref{FigFrucht2}.  In this case, $e=36$, $\alpha=12$ and $\beta=18$, the set ${\mathcal F}_{F^{clean}}$ has cardinality $2^{12}$ and its elements represent the non-isomorphic dessins d'enfants admitting the bipartite graph  $F^{clean}$. The elements of ${\mathcal F}_{F^{clean}}$ are given by the pairs $(\sigma_j, \tau)$, where
$$\tau=(1,20)(2,4)(5,7)(8,10)(11,13)(14,16)(17,19)(3,22)(6,23)\cdot$$ $$\cdot(9,29)(12,32)(15,33)(18,34)(21,35)(24,25)(26,28)(27,36)(30,31)$$
and the $2^{12}$ $\sigma_j's$ are of the form
$$(1,2,3)^{\pm 1} (4,5,6)^{\pm 1} (7,8,9)^{\pm 1} (10,11,12)^{\pm 1} (13,14,15)^{\pm 1} (16,17,18)^{\pm 1} (19,20,21)^{\pm 1} \cdot$$
$$\cdot (22,23,24)^{\pm 1} (25,26,27)^{\pm 1} (28,29,30)^{\pm 1} (31,32,33)^{\pm 1} (34,35,36)^{\pm 1}.$$

\begin{figure}[htbp]
\begin{center}
\includegraphics[width=3cm]{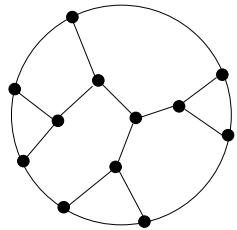}
\caption{{\bf Frucht's graph $F$}}
\label{FigFrucht1}
\end{center}
\end{figure}

\begin{figure}[htbp]
\begin{center}
\includegraphics[width=4cm]{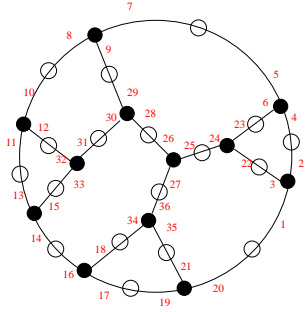}
\caption{{\bf The clean bipartite graph $F^{clean}$ associated to Frucht's graph}}
\label{FigFrucht2}
\end{center}
\end{figure}

The genus formula $g=1+(6-\gamma)/2$, where $\gamma$ denotes the number of faces of the dessin d'enfant, ensures that the possible genus are $g \in \{0,1,2,3\}$.
It is easy to see that $g=0$ is possible (Figure \ref{FigFrucht2} shows a dessin d'enfant of genus zero); this is provided with
$$\sigma_{0}=(1,2,3)(4,5,6)(7,8,9)(10,11,12)(13,14,15)(16,17,18)(19,20,21)\cdot$$
$$\cdot(22,23,24)(25,26,27)(28,29,30)(31,32,33)(34,35,36)$$ and the eight faces correspond to cycles of 
$$\tau\sigma_{0}=(1,21,36,25,22)(2,5,8,11,14,17,20)(3,23,4)(6,24,26,29,7)\cdot$$
$$\cdot(9,30,32,10)(12,33,13)(15,31,28,27,34,16)(18,35,19).$$

A dessin of genus one is provided by $$\sigma_{1}=(1,2,3)(4,5,6)(7,8,9)(10,11,12)(13,14,15)(16,17,18)(19,20,21)(22,23,24)\cdot$$ $$\cdot(25,26,27)(28,29,30)(31,32,33)(34,36,35)$$  in which case we have six faces and
$$\tau\sigma_{1}=(1,21,34,16,15,31,28,27,35,19,18,36,25,22)(2,5,8,11,14,17,20)\cdot$$
$$\cdot(3,23,4)(6,24,26,29,7)(9,30,32,10)(12,33,13).$$

A dessin of genus two is provided by $$\sigma_{2}=(1,2,3)(4,5,6)(7,8,9)(10,11,12)(13,14,15)(16,17,18)(19,20,21)(22,23,24)\cdot$$
$$\cdot(25,26,27)(28,29,30)(31,33,32)(34,36,35)$$  in which case we have four faces and
$$\tau\sigma_{2}=(1,21,34,16,15,32,10,9,30,33,13,12,31,28,27,35,19,18,36,25,22)\cdot$$
$$\cdot(2,5,8,11,14,17,20)(3,23,4)(6,24,26,29,7).$$

A dessin of genus three is provided by $$\sigma_{3}=(1,2,3)(4,5,6)(7,8,9)(10,11,12)(13,14,15)(16,17,18)(19,21,20)(22,24,23)\cdot$$
$$\cdot (25,26,27)(28,29,30)(31,33,32)(34,35,36)$$ in which case we have two faces and
$$\tau\sigma_{3}=(1,19,18,35,20,2,5,8,11,14,17,21,36,25,23,4,3,24,26,29,7,6,22)\cdot$$
$$\cdot (9,30,33,13,12,31,28,27,34,16,15,32,10).$$

All the above dessins are reflexive and are not dualizable (as there are vertices of odd degree).

\subsection{Example 4: The graph $K_{5}$}
Let $K_{5}^{clean}$ be the clean bipartite graph associated to the complete graph $K_{5}$ (see Figure \ref{FigK5}). With the given enumeration of the edges we have 
$$\tau=(1,8)(5,12)(9,16)(13,20)(4,17)(2,11)(7,18)(10,19)(3,14)(6,15)$$
and there are $7776$ choices for $\sigma$. In this case, $G_{K_{5}^{clean}}\cong {\mathfrak S}_{5}$ and $\theta(G_{K_{5}^{clean}})=\langle \eta_{1}, \eta_{2} \rangle$,
where $$\eta_{1}=(1,8)(2,5)(11,12)(3,6)(14,15)(4,7)(17,18),$$
$$\eta_{2}=(1,5,9,13,17)(8,12,16,20,4)(2,6,10,14,18)(11,15,19,3,7).$$

\begin{figure}[htbp]
\begin{center}
\includegraphics[width=5cm]{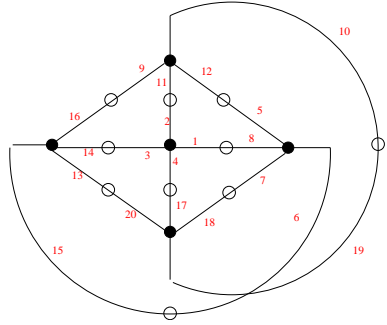}
\caption{{\bf The clean bipartite graph $K_{5}^{clean}$}}
\label{FigK5}
\end{center}
\end{figure}

The set ${\mathcal F}_{K_{5}^{clean}}/\theta(G_{K_{5}^{clean}})$ has $78$ elements, that is, there are exactly $78$ non-isomorphic dessins d'enfant whose bipartite graph is $K_{5}^{clean}$. 
Between these $78$ dessins d'enfants, there are exactly nine of genus $g=1$; these are given by the following choices for $\sigma$:
$$\sigma_{1}=(1,2,3,4)(5,6,7,8)(9,11,12,10)(13,14,16,15)(17,20,18,19)$$
$$\sigma_{2}=(1,2,3,4)(5,6,7,8)(9,11,12,10)(13,15,14,16)(17,18,20,19)$$
$$\sigma_{3}=(1,2,3,4)(5,6,7,8)(9,11,10,12)(13,14,16,15)(17,20,19,18)$$
$$\sigma_{4}=(1,2,3,4)(5,6,7,8)(9,11,10,12)(13,15,14,16)(17,19,18,20)$$
$$\sigma_{5}=(1,2,3,4)(5,6,8,7)(9,11,12,10)(13,14,15,16)(17,20,19,18)$$
$$\sigma_{6}=(1,2,3,4)(5,6,8,7)(9,11,12,10)(13,14,16,15)(17,20,19,18)$$
$$\sigma_{7}=(1,2,3,4)(5,6,8,7)(9,11,12,10)(13,15,14,16)(17,20,19,18)$$
$$\sigma_{8}=(1,2,3,4)(5,7,8,6)(9,12,11,10)(13,16,14,15)(17,19,18,20)$$
$$\sigma_{9}=(1,2,3,4)(5,7,6,8)(9,12,10,11)(13,16,15,14)(17,18,20,19)$$

Let us denote by $D_{j}$ the dessin d'enfant defined by $\sigma_{j}$, for $j=1,\ldots,9$.
The dessins d'enfant $D_{j}$, for $j=1,\ldots,7$, are non-uniform and those defined by $\sigma_{8}$ and $\sigma_{9}$ are regular. 
The two regular ones have passport $(4^5;2^{10};4^5)$ and they form a chiral pair. The automorphism group is isomorphic to ${\mathbb Z}_{5} \rtimes {\mathbb Z}_{4}$ and they are defined over the same elliptic curve $y^{2}=x^{4}-1$. 
The dessin $D_{5}$ is the only one whose monodromy group is isomorphic to $({\mathbb Z}_{4}^{4} \rtimes {\mathcal A}_{5}) \rtimes {\mathbb Z}_{4}$ (whose group of automorphisms is ${\mathbb Z}_{4}$); this again is over the elliptic curve $y^{2}=x^{4}-1$. This dessin is reflexive and its passport is $(4^5;2^{10};3^4,8)$.
The genus one dessin d'enfant $D_{2}$ is the only with monodromy group of order $26.336.378.880.000$ and it has trivial group of automorphisms. This dessin is reflexive and its passport is $(4^5;2^{10};3^3,4,7)$.
The other five dessins have monodromy group of order $1.857.945.600$ and group of automorphisms isomorphic to ${\mathbb Z}_{2}$. The dessins $D_{1}$ and $D_{3}$ (respectively, by $D_{6}$ and by $D_{7}$) form a chiral pair. The passport of $D_{1}$ and $D_{3}$ is $(4^5;2^{10};3^2,4^2,6)$ and the passport of $D_{6}$ and $D_{7}$ is $(4^5;2^{10};3^4,8)$.  The dessin $D_{4}$ is reflexive and its passport is $(4^5;2^{10};3^2,4,5^2)$. All these dessins are not dualizable (as there are vertices of odd degree).

\subsection{Example 5: The double-prism graph}\label{Sec:doble-prisma}
Let us now consider the bipartite graph ${\mathcal G}$ as shown in Figure \ref{DoublePrism} with the given enumeration of the edges. With the given enumeration,
$$\tau=(1,5)(2,6)(3,7)(4,8)(9,13)(10,14)(11,15)(12,16)(17,21)(18,22)(19,23)(20,24)$$
and there are $6^6$ choices for $\sigma$. The dessins d'enfants admitting the above bipartite graph are of genus $g \in \{0,1,2,3\}$ and there are exactly $5946$ non-isomorphic ones: two non-isomorphic dessins of genus zero (both are dualizable), $79$ of genus one ($22$ of them are dualizable), $1849$ of genus two ($121$ of them are dualizable) and $4016$ of genus three (only $33$ being dualizable). 

Of these genus one dessins d'enfants, there are exactly $13$ with passport $(4^6;2^{12};3^2,4^2,5^2)$, all of them have monodromy group of order $980995276800$, exactly $3$ of them which are reflexive and the others are chirals pairs. Two of these dessins, say ${\mathcal D}_{1}$ and ${\mathcal D}_{2}$ are provided by the permutations
 $$\sigma_{1}=(1,2,3,4)(5,13,17,24)(6,18,21,14)(7,15,22,19)(8,23,20,16)(9,12,11,10)$$
$$\sigma_{2}=(1,2,4,3)(5,13,17,24)(6,14,18,21)(7,19,22,15)(8,16,20,23)
(9,11,10,12)$$
$$\tau_{1}=\tau_{2}=(1,5)(2,6)(3,7)(4,8)(9,13)(10,14)(11,15)(12,16)(17,21)(18,22)
(19,23)(20,24).$$

\begin{figure}[htbp]
\begin{center}
\includegraphics[width=4cm]{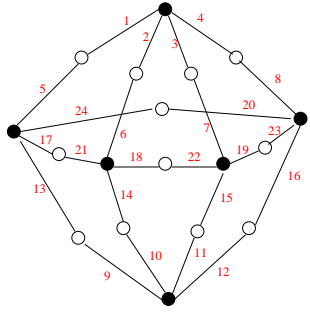}
\caption{{\bf The double-prism graph}}
\label{DoublePrism}
\end{center}
\end{figure}

Using GAP \cite{GAP} one can check that both of them have isomorphic monodromy group; but, the dessin d'enfant ${\mathcal D}_{1}$ is chiral (with group of automorphisms isomorphic to ${\mathbb Z}_{2}$) and ${\mathcal D}_{2}$ is reflexive (with trivial group of automorphisms); so they cannot be in the same Galois orbit.

\subsection{Example 6: dessins d'enfants defined by bipartite graphs with passport $(3^3;3^3)$}
There are, up to isomorphisms, exactly three bipartite graphs with passport $(3^3;3^3)$ ($e=9$, $\alpha=\beta=3$); these being the bipartite graphs $K_{3,3}$, $D_{3,3}$ and $C_{3,3}$ shown in Figure \ref{Fig1}. In each case we fix a labelling of the black vertices as $b_{1}, b_{2}, b_{3}$, white vertices as $w_{1}, w_{2}, w_{3}$ and the nine edges with numbers in $\{1,\ldots, 9\}$ without repeating. For each ${\mathcal G} \in \{K_{3,3}, D_{3,3}, C_{3,3}\}$, the cardinality of ${\mathcal F}_{\mathcal G}$ is $64$. We proceed, in each case, to describe all non-isomorphic dessins d'enfants whose bipartite graph is isomorphic to ${\mathcal G}$. As the black vertices have odd degree, none of the dessins admitting these bipartite graphs is dualizable.

\begin{figure}[htbp]
\begin{center}
\includegraphics[width=7cm]{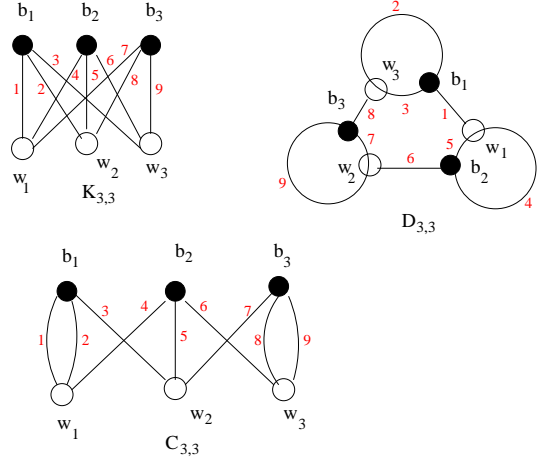}
\caption{{\bf Bipartite graphs with passport $(3^3;3^3)$}}
\label{Fig1}
\end{center}
\end{figure}

\subsubsection{The complete bipartite graph $K_{3,3}$}\label{ejemplo21}
In this case, $G_{K_{3,3}}\cong D_{3} \times D_{3}$ and $\theta(G_{K_{3,3}})=\langle \eta_{1}, \eta_{2},\eta_{3},\eta_{4} \rangle$,
where $$\eta_{1}=(1,4)(2,5)(3,6),\; \eta_{2}=(1,4,7)(2,5,8)(3,6,9),\; \eta_{3}=(1,2)(4,5)(7,8), \; \eta_{4}=(1,2,3)(4,5,6)(7,8,9).$$

The set ${\mathcal F}_{K_{3,3}}/\theta(G_{K_{3,3}})$ has four elements, these are represented by the pairs
$(\sigma_{1},\tau_{1})$, $(\sigma_{1},\tau_{2})$, $(\sigma_{2},\tau_{1})$ and $(\sigma_{2},\tau_{2})$,
where
$$\sigma_{1}=(1,2,3)(4,5,6)(7,8,9), \; \sigma_{2}=(1,2,3)(4,5,6)(7,9,8),$$
$$\tau_{1}=(1,4,7)(2,5,8)(3,6,9), \; \tau_{2}=(1,4,7)(2,5,8)(3,9,6),$$
that is, there are exactly four non-isomorphic dessins d'enfant with $K_{3,3}$ as bipartite graph. The $G_{\mathcal G}$-orbit of $(\sigma_{1},\tau_{1})$ has length $4$, the orbits of $(\sigma_{1},\tau_{2})$ and also of $(\sigma_{2},\tau_{1})$ have lengths $12$ and that of $(\sigma_{2},\tau_{2})$ has length $36$.
These four non-isomorphic dessins have the following properties:
\begin{enumerate} 
\item The dessin d'enfant with monodromy group $E_{1}=\langle \sigma_{1},\tau_{1}\rangle \cong {\mathbb Z}_{3}^{2}$ has genus $g=1$ (since $\tau_{1}\sigma_{1}=(1,5,9)(2,6,7)(3,4,8)$), its passport is
$(3^3;3^3;3^3)$, its 
$\theta(G_{K_{3,3}})$-stabilizer is given by the group $\langle (1,3,2)(4,6,5)(7,9,8), (1,7,4)(2,8,5)(3,9,6) \rangle \cong {\mathbb Z}_{3}^{2}$ (its group of automorphisms)
and the underlying Riemann surface of genus one is defined by the Fermat curve of degree three: $x^{3}+y^{3}+z^{3}=0$. In this case, the group of automorphisms ${\mathbb Z}_{3}^{2}$ is generated by the order three automorphisms
$a_{1}[x:y:z]:=[\omega_{3}x:y:z]$ and $a_{2}[x:y:z]:=[x:\omega_{3}y:z]$. This dessin is reflexive.

\item The dessin d'enfant with monodromy group $E_{2}=\langle \sigma_{1},\tau_{2}\rangle \cong {\mathbb Z}_{3}^{3} \rtimes {\mathbb Z}_{3}$ has genus  $g=2$ (since $\tau_{2}\sigma_{1}=(1,5,9,4,8,3,7,2,6)$), its passport is $(3^3;3^3;9)$, its 
$\theta(G_{K_{3,3}})$-stabilizer is $\langle (1,7,4)(2,8,5)(3,9,6) \rangle \cong {\mathbb Z}_{3}$ (its group of automorphisms) and the underlying Riemann surface of genus two is defined by $y^{3}=(x-1)(x^2+x+1)^{2}$.

\item The dessin d'enfant with monodromy group $E_{3}=\langle \sigma_{2},\tau_{1}\rangle \cong {\mathbb Z}_{3}^{3} \rtimes {\mathbb Z}_{3}$ has genus $g=2$ (since $\tau_{1}\sigma_{2}=(1,5,7,2,6,8,3,4,9)$), its passport is $(3^3;3^3;9)$, its 
$\theta(G_{K_{3,3}})$-stabilizer is $\langle (1,3,2)(4,6,5)(7,9,8) \rangle \cong   {\mathbb Z}_{3}$ (its group of automorphisms) and the same underlying Riemann surface of genus two as above ($E_{2}$ and $E_{3}$ are chirals). 

\item The dessin d'enfant with monodromy group $E_{4}=\langle \sigma_{2},\tau_{2}\rangle \cong {\mathcal A}_{9}$ has genus $g=1$ (since $\tau_{2}\sigma_{2}=(1,5,7,2,6)(3,8)(4,9)$), its passport is
$(3^3;3^3;2^2,5)$, it  is non-uniform, and it has trivial $\theta(G_{K_{3,3}})$-stabilizer (its group of automorphisms). This dessin is reflexive. As the dessin d'enfant has trivial group of automorphisms, it can be defined over its field of moduli. As the dessin d'enfant is unique, its field of moduli is ${\mathbb Q}$. In particular, the genus one Riemann surface it defines can be defined over ${\mathbb Q}$. 
\end{enumerate}

\s
\begin{rema}[On Wilson's operations]
In this example there are four Wilson's operations: $H(1,1)=I$ (identity), $H(1,2), H(2,1)$ and $H(2,2)=H(1,2)\circ H(2,1)$ (each one of order two); they form a copy of ${\mathbb Z}_{2}^{2}$ inside ${\mathfrak S}_{9}$. This group keeps invariant ${\mathcal F}_{K_{3,3}}$ and, moreover, keeps invariant each of the $4$ $G_{\mathcal G}$-orbits. The action is transitive on the $G_{\mathcal G}$-orbit of $(\sigma_{1},\tau_{1})$, produces $3$ orbits in each of the $G_{\mathcal G}$-orbits of $(\sigma_{1},\tau_{2})$ and of $(\sigma_{2},\tau_{1})$ and produces $9$ orbits in the $G_{\mathcal G}$-orbit of $(\sigma_{2},\tau_{2})$. In particular, there are isomorphic dessins d'enfants which are not equivalent under Wilson's operations.
\end{rema}

\subsubsection{The bipartite graph $D_{3,3}$}
In this case, $G_{D_{3,3}}\cong {\mathbb Z}_{2} \times {\mathcal A}_{4}$ and $\theta(G_{D_{3,3}})=\langle \eta_{1}, \eta_{2} \rangle$,
where $$\eta_{1}=(1,8,6)(2,9,4)(3,7,5),\; \eta_{2}=(2,3).$$

The set ${\mathcal F}_{D_{3,3}}/\theta(G_{D_{3,3}})$ has four elements, these are represented by the pairs
$(\sigma,\tau_{1})$, $(\sigma,\tau_{2})$, $(\sigma,\tau_{3})$ and $(\sigma,\tau_{4})$, 
where
$$\sigma=(1,2,3)(4,5,6)(7,8,9),$$
$$\tau_{1}=(1,4,5)(2,3,8)(6,7,9), \; \tau_{2}=(1,4,5)(2,3,8)(6,9,7),$$
$$\tau_{3}=(1,4,5)(2,8,3)(6,7,9), \; \tau_{4}=(1,5,4)(2,8,3)(6,9,7)$$
that is, there are exactly four non-isomorphic dessins d'enfant with $D_{3,3}$ as bipartite graph. The $G_{\mathcal G}$-orbit of $(\sigma,\tau_{1})$ has length $24$, the orbit of $(\sigma,\tau_{2})$ has length $8$, the orbit of $(\sigma,\tau_{3})$ have lengths $24$ and that of $(\sigma,\tau_{4})$ has length $8$.
These four non-isomorphic dessins have the following properties:
\begin{enumerate} 
\item The dessin d'enfant with monodromy group $F_{1}=\langle \sigma,\tau_{1}\rangle \cong {\mathcal A}_{9}$ has genus $g=1$ (since $\tau_{1}\sigma=(1,5,2)(3,9,4,6,8)(7)$), its passport is $(3^3;3^3;1,3,5)$ and its 
$\theta(G_{D_{3,3}})$-stabilizer is trivial (its group of automorphisms).

\item The dessin d'enfant with monodromy group $F_{2}=\langle \sigma,\tau_{2}\rangle \cong {\mathbb Z}_{3}^{2} \rtimes {\mathbb Z}_{3}$ has genus  $g=1$ (since $\tau_{2}\sigma=(1,5,2)(3,9,8)(4,6,7)$), its passport is
$(3^3;3^3;3^3)$ and its 
$\theta(G_{D_{3,3}})$-stabilizer $\langle (1,8,6)(2,9,4)(3,7,5) \rangle \cong {\mathbb Z}_{3}$ (its group of automorphisms).

\item The dessin d'enfant with monodromy group $F_{3}=\langle \sigma,\tau_{3}\rangle \cong {\mathcal A}_{9}$ has genus $g=1$ (since $\tau_{3}\sigma=(1,5,2,9,4,6,8)(3)(7)$), its passport is $(3^3;3^3;1^2,7)$ and its 
$\theta(G_{D_{3,3}})$-stabilizer is trivial (its group of automorphisms).

\item The dessin d'enfant with monodromy group $F_{4}=\langle \sigma,\tau_{4}\rangle \cong {\mathbb Z}_{3}^{2} \rtimes {\mathbb Z}_{3}$ has genus $g=0$ (since $\tau_{4}\sigma=(1,6,8)(2,7,4)(3)(5)(9)$), its passport is $(3^3;3^3;1^3,3^2)$ and its
 $\theta(G_{D_{3,3}})$-stabilizer $\langle (1,8,6)(2,7,5)(3,9,4))\rangle \cong  {\mathbb Z}_{3}$ (its group of automorphisms). 
 \end{enumerate}

All these dessins are reflexive.

\subsubsection{The bipartite graph $C_{3,3}$}
In this case, $G_{C_{3,3}}\cong D_{8}$ and $\theta(G_{C_{3,3}})=\langle \eta_{1}, \eta_{2} \rangle$,
where $$\eta_{1}=(4,6)(3,7)(1,8)(2,9),\; \eta_{2}=(1,2).$$

The set ${\mathcal F}_{C_{3,3}}/\theta(G_{C_{3,3}})$ has eight elements, these are represented by the pairs
$(\sigma,\tau_{1})$, $(\sigma,\tau_{2})$, $(\sigma,\tau_{3})$, $(\sigma,\tau_{4})$, $(\sigma,\tau_{5})$, $(\sigma,\tau_{6})$, $(\sigma,\tau_{7})$ and $(\sigma,\tau_{8})$,
where
$$\sigma=(1,2,3)(4,5,6)(7,8,9),$$
$$\tau_{1}=(1,2,4)(3,5,7)(6,8,9), \; \tau_{2}=(1,2,4)(3,5,7)(6,9,8),$$
$$\tau_{3}=(1,2,4)(3,7,5)(6,8,9), \; \tau_{4}=(1,2,4)(3,7,5)(6,9,8)$$
$$\tau_{5}=(1,4,2)(3,5,7)(6,8,9), \; \tau_{6}=(1,4,2)(3,5,7)(6,9,8),$$
$$\tau_{7}=(1,4,2)(3,7,5)(6,9,8), \; \tau_{8}=(1,4,2)(3,7,5)(6,8,9)$$
that is, there are exactly eight non-isomorphic dessins d'enfant with $C_{3,3}$ as bipartite graph. The $G_{\mathcal G}$-orbit of each $(\sigma,\tau_{j})$ has length $8$.
These eight non-isomorphic dessins have the following properties:
\begin{enumerate} 
\item The dessin d'enfant with monodromy group $G_{1}=\langle \sigma,\tau_{1}\rangle \cong (({\mathbb Z}_{3} \times ({\mathbb Z}_{3}^{2} \rtimes {\mathbb Z}_{2})) \rtimes {\mathbb Z}_{2}) \rtimes {\mathbb Z}_{3}$ has genus $g=2$ (since $\tau_{1}\sigma=(1,3,6,9,4,2,5,8,7)$), its passport is $(3^3;3^3;9)$ and its  
$\theta(G_{C_{3,3}})$-stabilizer is trivial (its group of automorphisms). This dessin is reflexive.

\item The dessin d'enfant with monodromy group $G_{2}=\langle \sigma,\tau_{2}\rangle \cong ({\mathbb Z}_{3}^{2} \rtimes Q_{8}) \rtimes {\mathbb Z}_{3}$ has genus  $g=1$ (since $\tau_{2}\sigma=(1,3,6,7)(2,5,8,4)(9)$), its passport is $(3^3;3^3;1,4^2)$ and its $\theta(G_{C_{3,3}})$-stabilizer is trivial (its group of automorphisms). This dessin is chiral to the one defined by $G_5$.

\item The dessin d'enfant with monodromy group $G_{3}=\langle \sigma,\tau_{3}\rangle \cong {\rm PSL}_{2}(8)$ has genus $g=2$ (since $\tau_{3}\sigma=(1,3,8,7,6,9,4,2,5)$), its passport is 
$(3^3;3^3;9)$ and its 
$\theta(G_{C_{3,3}})$-stabilizer is trivial (its group of automorphisms). This dessin is reflexive.

\item The dessin d'enfant with monodromy group $G_{4}=\langle \sigma,\tau_{4}\rangle \cong ({\mathbb Z}_{3}^{2} \rtimes Q_{8}) \rtimes {\mathbb Z}_{3}$ has genus $g=1$ (since $\tau_{4}\sigma=(1,3,8,4,2,5)(6,7)(9)$), its passport is $(3^3;3^3;1,2,6)$ and its
 $\theta(G_{C_{3,3}})$-stabilizer is trivial (its group of automorphisms).  This dessin is reflexive.
 
\item The dessin d'enfant with monodromy group $G_{5}=\langle \sigma,\tau_{5}\rangle \cong ({\mathbb Z}_{3}^{2} \rtimes Q_{8}) \rtimes {\mathbb Z}_{3}$ has genus $g=1$ (since $\tau_{5}\sigma=(1,5,8,7)(3,6,9,4)(2)$), its passport is $(3^3;3^3;1,4^2)$ and its 
$\theta(G_{C_{3,3}})$-stabilizer is trivial (its group of automorphisms). This dessin is chiral to the one defined by $G_2$.

\item The dessin d'enfant with monodromy group $G_{6}=\langle \sigma,\tau_{6}\rangle \cong {\rm PSL}_{2}(8)$ has genus  $g=1$ (since $\tau_{6}\sigma=(1,5,8,4,3,6,7)(2)(9)$), its passport is
$(3^3;3^3;1^2,7)$ and its 
$\theta(G_{C_{3,3}})$-stabilizer is trivial (its group of automorphisms). This dessin is reflexive.

\item The dessin d'enfant with monodromy group $G_{7}=\langle \sigma,\tau_{7}\rangle \cong (({\mathbb Z}_{3} \times ({\mathbb Z}_{3}^{2} \rtimes {\mathbb Z}_{2})) \rtimes {\mathbb Z}_{2}) \rtimes {\mathbb Z}_{3}$ has genus $g=0$ (since $\tau_{7}\sigma=(1,5)(3,8,4)(6,7)(2)(9)$), its passport is $(3^3;3^3;1^2,2^2,3)$ and its 
$\theta(G_{C_{3,3}})$-stabilizer is trivial (its group of automorphisms). This dessin is reflexive.

\item The dessin d'enfant with monodromy group $G_{8}=\langle \sigma,\tau_{8}\rangle \cong ({\mathbb Z}_{3}^{2} \rtimes Q_{8}) \rtimes {\mathbb Z}_{3}$ has genus $g=1$ (since $\tau_{8}\sigma=(1,5)(3,8,7,6,9,4)(2)$), its passport is $(3^3;3^3;1,2,6)$ and its
 $\theta(G_{C_{3,3}})$-stabilizer is trivial (its group of automorphisms).  This dessin is reflexive.

\end{enumerate}

\subsubsection{}
Table \ref{tabla1} summarizes all the above. We may see the following facts.

\begin{enumerate}
\item There are exactly two non-isomorphic genus zero dessins d'enfants whose bipartite graphs have passport $(3^3;3^3)$; so each one has field of moduli ${\mathbb Q}$. For instance, the one corresponding to $D_{3,3}$ has Belyi map 
$${\mathfrak B}(z)=\frac{(z^{3}-1)^{3}}{3(\omega_{3}-1)z^{3}(z^{3}-(\omega_{3}+1))}$$
where $\omega_{3}=e^{2 \pi i/3}$ (so it is defined over the extension of degree two ${\mathbb Q}(\omega_{3})$). If $\sigma(\omega_{3})=\omega_{3}^{2}$, then 
$${\mathfrak B}^{\sigma}={\mathfrak B} \circ T, \quad T(z)=1/z.$$

In this way, noticing that $\{I,T\}$ is a Weil's co-cycle with respect to the Galois extension ${\mathbb Q}(\omega_{3})/{\mathbb Q}$, we may see that this dessin d'enfant is in fact definable over ${\mathbb Q}$.

\item There are nine non-isomorphic genus one dessins d'enfants whose bipartite graphs have passport $(3^3;3^3)$. With the exceptions of those with monodromy group $G_{2}$, $G_{4}$, $G_{5}$ and $G_{8}$, each of them has field of moduli ${\mathbb Q}$.
The two dessins $G_{2}$ and $G_{5}$ (respectively, $G_{4}$ and $G_{8}$) form a Galois orbit and they are chirals.

\item There are three genus two dessins d'enfants, up to strong-isomorphisms, whose bipartite graphs have passport $(3^3;3^3)$, all of them with field of moduli ${\mathbb Q}$ and reflexive. 

\item All those dessins in Table \ref{tabla1} with trivial group of automorphisms together the regular one are definable over their field of moduli.

\end{enumerate}

{\tiny 
\begin{table}[htp]
\centering
\caption{Dessins d'enfants with bipartite graph of passport $(3^3;3^3)$ up to isomorphisms: FOM=field of moduli, FOD=field of definition.}

\rotatebox{00}{
\begin{tabular}{|c|c|c|c|c|c|c|c|}
\hline
Graph & Genus & Monodromy & Aut & Passport & Regular & FOM & FOM=FOD\\
\hline
$D_{3,3}$ & $0$ & $F_{4} \cong {\mathbb Z}_{3}^{2} \rtimes {\mathbb Z}_{3}$ & ${\mathbb Z}_{3}$ & $(3^3;3^3;1^3,3^2)$ & N & ${\mathbb Q}$ & Y\\
$C_{3,3}$ & $0$ & $G_{8} \cong ({\mathbb Z}_{3} \times ({\mathbb Z}_{3}^{2} \rtimes {\mathbb Z}_{2})\rtimes{\mathbb Z}_{2})\rtimes{\mathbb Z}_{3}$ & $\{I\}$ & $(3^3;3^3;1^2,2^2,3)$ & N & ${\mathbb Q}$ & Y\\
\hline
$K_{3,3}$ & $1$ & $E_{1} \cong {\mathbb Z}_{3}^{2}$ & ${\mathbb Z}_{3}^{2}$ & $(3^3;3^3;3^3)$ & Y & ${\mathbb Q}$ & Y\\
$K_{3,3}$ & $1$ & $E_{4} \cong {\mathcal A}_{9}$ & $\{I\}$ & $(3^3;3^3;2^2,5)$ & N & ${\mathbb Q}$ & Y\\
$D_{3,3}$ & $1$ & $F_{1} \cong {\mathcal A}_{9}$ & $\{I\}$ & $(3^3;3^3;1,3,5)$ & N & ${\mathbb Q}$ & Y\\
$D_{3,3}$ & $1$ & $F_{3} \cong {\mathcal A}_{9}$ & $\{I\}$ & $(3^3;3^3;1^2,7)$ & N & ${\mathbb Q}$ & Y\\
$D_{3,3}$ & $1$ & $F_{2} \cong {\mathbb Z}_{3}^{2} \rtimes {\mathbb Z}_{3}$ & ${\mathbb Z}_{3}$ & $(3^3;3^3;3^3)$ & N & ${\mathbb Q}$ & ?\\
$C_{3,3}$ & $1$ & $G_{2} \cong ({\mathbb Z}_{3}^{2} \rtimes Q_{8})\rtimes{\mathbb Z}_{3}$ & $\{I\}$ & $(3^3;3^3;1,4^2)$ & N & $[{\rm FOM}:{\mathbb Q}]=2$ & Y\\
$C_{3,3}$ & $1$ & $G_{5} \cong ({\mathbb Z}_{3}^{2} \rtimes Q_{8})\rtimes{\mathbb Z}_{3}$ & $\{I\}$ & $(3^3;3^3;1,4^2)$ & N & $[{\rm FOM}:{\mathbb Q}]=2$& Y\\
$C_{3,3}$ & $1$ & $G_{6} \cong {\rm PSL}_{2}(8)$ & $\{I\}$ & $(3^3;3^3;1^2,7)$ & N  & ${\mathbb Q}$ &Y\\
$C_{3,3}$ & $1$ & $G_{4} \cong ({\mathbb Z}_{3}^{2} \rtimes Q_{8})\rtimes{\mathbb Z}_{3}$ & $\{I\}$ & $(3^3;3^3;1,2,6)$ & N & $[{\rm FOM}:{\mathbb Q}]=2$ & Y\\
$C_{3,3}$ & $1$ & $G_{8} \cong ({\mathbb Z}_{3}^{2} \rtimes Q_{8})\rtimes{\mathbb Z}_{3}$ & $\{I\}$ & $(3^3;3^3;1,2,6)$ & N & $[{\rm FOM}:{\mathbb Q}]=2$ & Y\\
\hline
$K_{3,3}$ & $2$ & $E_{3} \cong {\mathbb Z}_{3}^{3}\rtimes{\mathbb Z}_{3}$ & ${\mathbb Z}_{3}$ & $(3^3;3^3;9)$ & N & ${\mathbb Q}$ & ?\\
$C_{3,3}$ & $2$ & $G_{1} \cong ({\mathbb Z}_{3} \times ({\mathbb Z}_{3}^{2} \rtimes {\mathbb Z}_{2})\rtimes{\mathbb Z}_{2})\rtimes{\mathbb Z}_{3}$ & $\{I\}$ & $(3^3;3^3;9)$ & N & ${\mathbb Q}$ & Y\\
$C_{3,3}$ & $2$ & $G_{3} \cong {\rm PSL}_{2}(8)$ & $\{I\}$ & $(3^3;3^3;9)$ & N & ${\mathbb Q}$ & Y \\
\hline
\end{tabular}
}
\label{tabla1}
\end{table}
}



\end{document}